\documentclass[11pt]{amsproc}

\usepackage{mathtext}
\usepackage[cp1251]{inputenc}

\usepackage{bm}

\usepackage[dvips]{graphicx}
\usepackage{amsmath}
\usepackage{amssymb}
\usepackage{amsxtra}

\usepackage{epsfig}
\usepackage{epic}
\usepackage{eepic}
\usepackage{graphics}
\usepackage{graphicx}
\usepackage{subfigure}

\usepackage{caption}
\def\N{{{\Bbb N}}}
\def\Z{{{\Bbb Z}}}
\def\T{{{\Bbb T}}}
\def\R{{\Bbb R}}

\def\l{{\lambda }}
\def\a{{\alpha }}
\def\D{{\Delta }}

\def\a{{\alpha}}

\def\d{{\delta}}

\def\vp{{\varphi}}

\def\g{{\gamma }}
\def\w{{\omega }}
\def\L{{\mathcal{L} }}
\def\G{{\Gamma }}

\def\B{{\mathcal{B} }}

\def\La{{\Lambda }}
\def\I{{\mathcal{I} }}
\def\V{{\mathcal{V} }}
\def\Np{{\mathcal{N} }}
\def\Tp{{\mathcal{T} }}
\def\Pp{{\mathcal{P} }}
\def\B{{\mathcal{B} }}

\def\n{{\bm n}}
\def\k{{\bm k}}
\def\ll{{\bm l}}
\def\x{{\bm x}}

\def\){\right)}
\def\({\left(}

\def\spann{\operatorname{span}}

\numberwithin{equation}{section}
\newtheorem{corollary}{Corollary}[section]
\newtheorem{lemma}{Lemma}[section]
\newtheorem{theorem}{Theorem}[section]
\newtheorem{proposition}{Proposition}[section]
\newtheorem{remark}{Remark}[section]

\par

\begin{document}

\title[On $L_p$-error of bivariate polynomial interpolation on the square]{On $L_p$-error of bivariate polynomial \\ interpolation on the square}

\author[Yurii
Kolomoitsev]{Yurii
Kolomoitsev$^{\text{a, b, 1, *}}$}
\address{Universit\"at zu L\"ubeck,
Institut f\"ur Mathematik,
Ratzeburger Allee 160,
23562 L\"ubeck}
\email{kolomoitsev@math.uni-luebeck.de, kolomus1@mail.ru}

\author[Tetiana
Lomako]{Tetiana
Lomako$^{\text{a, b, 2}}$}
\address{Institute of Applied Mathematics and Mechanics of NAS of Ukraine,
General Batyuk Str.~19, Slov’yans’k, Donetsk region, Ukraine, 84116}
\email{tlomako@yandex.ru}

\author[J\"urgen~Prestin]{J\"urgen~Prestin$^\text{a, 2}$}
\address{Universit\"at zu L\"ubeck,
Institut f\"ur Mathematik,
Ratzeburger Allee 160,
23562 L\"ubeck}
\email{prestin@math.uni-luebeck.de}

\thanks{$^\text{a}$Universit\"at zu L\"ubeck,
Institut f\"ur Mathematik,
Ratzeburger Allee 160,
23562 L\"ubeck, Germany}

\thanks{$^\text{b}$Institute of Applied Mathematics and Mechanics of NAS of Ukraine, General Batyuk Str.~19, Slov’yans’k, Donetsk region, Ukraine, 84116}

\thanks{$^1$Supported by H2020-MSCA-IF-2015 Project number 704030   (AFFMA: "\!Approximation of Functions and Fourier Multipliers and their applications").}

\thanks{$^2$Supported by   H2020-MSCA-RISE-2014 Project number 645672  (AMMODIT: "\!Approximation Methods for Molecular Modelling and Diagnosis Tools").}

\thanks{$^*$Corresponding author}

\thanks{E-mail address: kolomoitsev@math.uni-luebeck.de, kolomus1@mail.ru}

\date{\today}

\begin{abstract}
We obtain estimates of the $L_p$-error of the bivariate polynomial interpolation on the Lissajous-Chebyshev node points for wide classes of functions including non-smooth functions of bounded variation in the sense of Hardy-Krause. The results show that $L_p$-errors of polynomial interpolation on the Lissajous-Chebyshev nodes have almost the same behavior as the polynomial interpolation in the case of the tensor product Chebyshev grid.
\end{abstract}

\maketitle

\section{Introduction}

Let $B(D)$ be the set of all real-valued bounded measurable functions $f$ on some non-empty set $D\subset \R^d$.
Denote $\k=(k_1,\dots,k_d)\in \Z_+^d$, $\n=(n_1,\dots,n_d)\in \N^d$, and $\x=(x_1,\dots,x_d)\in D$.
Let $X_\n=\{\x_\k\}_{\k\in \I_{\n}}$ be a  set of distinct points $\x_\k=(x_{1}^{(\k)},\dots, x_{d}^{(\k)})\in D$, where $\I_{\n}\subset \Z_+^d$ is the index set related to $X_\n$; let $\Pi_\n$ be some abstract space of polynomial functions  on $D$ (e.g. algebraic or trigonometric polynomials of order at most $\n$ in some sense).

The polynomial interpolation problem associated to $X_\n$ and $\Pi_\n$ is to find for each $f\in B(D)$ a polynomial $P\in \Pi_\n$ such that
\begin{equation}\label{Prob}
  P(\x_{\k})=f(\x_{\k})\quad\text{for all}\quad \k\in \I_{\n}.
\end{equation}
If this problem is solvable, then one can
construct the so-called Lagrange interpolation polynomial
\begin{equation*}\label{L}
  L(f,X_\n,\x)=\sum_{\k\in \I_{\n}}f(\x_\k)\ell_{\k}(X_\n,\x),
\end{equation*}
where $\ell_{\k}(X_\n,\x)\in \Pi_\n$ are  fundamental Lagrange
polynomials, i.e.
$
  \ell_{\k}(X_\n,\x_\ll)=\d_{\ll,\k}$ for all $\ll,\k\in \I_\n
$
($\d_{\ll,\k}$ is the Kronecker delta).
In this case, we have
\begin{equation}\label{L=f}
  L(f,X_\n,\x_\k)=f(\x_\k)\quad\text{for all}\quad \k\in \I_\n.
\end{equation}

Everywhere below we will suppose that~\eqref{Prob} is solvable and
\begin{equation}\label{LP}
 L(P,X_\n,\x)=P(\x)\quad\text{for all}\quad P\in \Pi_\n.
\end{equation}

One of the main problems in multivariate polynomial interpolation is the choice of a suitable set of node points $X_{\n}$,
which should provide "\!good" properties of the interpolation polynomials $L(f,X_\n,\x)$ depending on considered tasks.
In the bivariate
interpolation on $[-1,1]^2$, efficient node points were introduced by Morrow and Patterson~\cite{MP},
Xu~\cite{Xu1996} (see also~\cite{Ha}), Caliari, De Marchi, Vianello~\cite{CMV}.
The set of points introduced in~\cite{CMV} are known as the so-called Padua points.
It turned out that the  Padua points are a promising set of nodes for bivariate polynomial interpolation.
In particular, this point set allows unique interpolation in the appropriate space of polynomials; the Lebesgue constant of the interpolation problem on the Padua points grows only logarithmically with respect to the total degree of the interpolating polynomial; the Padua points can be characterized as a set of node points of a particular Lissajous curve~\cite{BDeMVXu2006}.
The last property can be applied in a new medical imaging technology called Magnetic
Particle Imaging. The main reason of this application is that the data acquisition path of the scanning device is commonly described by the Lissajous curves (see e.g.~\cite{KBSWGBB}). Under investigation of the bivariate polynomial interpolation of data values on these curves it was gained several extensions of the generating curve approach of the Padua points  in~\cite{E} and~\cite{ErbKaethnerAhlborgBuzug2015} (see also the survey~\cite{ErbKaethnerDenckerAhlborg2015}).

Let us recall that in the bivariate setting the Lissajous curve $\g_{\n}:\R\to J^2$, $J=[-1,1]$,  is given by
$$
\g_{\n}(t)=\(\cos(n_2t), \cos(n_1t)\),
$$
where $n_1$ and $n_2$ are relatively prime.
If we sample the curve $\g_{\n}$ along $n_1n_2+1$ equidistant points on $[0,\pi]$, then we get the set
$$
{\rm LC}_{\n}=\left\{\g_{\n}\(\frac{\pi k}{n_1n_2}\),\,k=0,\dots,n_1n_2\right\}
$$
called the Lissajous-Chebyshev node points of the (degenerate) Lissajous curve (see~\cite{E}). Note that in the case $n_2=n_1+1$ the Lissajous-Chebyshev node points ${\rm LC}_{\n}$ coincide with the Padua points ${\rm Pad}_{n_1}$ (see~\cite{CMV}, \cite{BDeMVXu2006}).

The multivariate Lagrange interpolation on the Lissajous-Chebyshev node points was studied in detail in the recent papers~\cite{E}, \cite{DE}, and~\cite{DEKL}. In particular, the following estimate of the error of approximation  by the interpolation polynomials $L(f,{\rm LC}_{\n},\cdot)$ was obtained in~\cite{DEKL}:
%
%
\noindent\emph{if $\frac{\partial^r}{\partial^rx_1} f$,  $\frac{\partial^s}{\partial^sx_2}f\in C(J^2)$, then}
\begin{equation}\label{int1}
  \Vert f-L(f,{\rm LC}_{\n},\cdot)\Vert_{L_\infty(J^2)}=\mathcal{O}(\log n_1\log n_2)\(\frac1{{n_1}^{r}}+\frac1{{n_2}^{s}}\).
\end{equation}

\smallskip

Note that the polynomial $L(f,{\rm LC}_{\n},\cdot)$ has total degree at most $n_1+n_2-1$. Precisely, it has degree $n_1-1$ in $x_1$ and degree $n_2$ in $x_2$ (see also Theorem~\ref{thErb} and Theorem~\ref{thDEKL} below).

Concerning approximation in the weighted $L_{p,w}(J^2)$ space with the Chebyshev weight $w(x_1,x_2)=1/{(\sqrt{1-x_1^2}\sqrt{1-x_2^2})^\frac1p}$, it is only known  convergence (see~\cite{E}):

\smallskip

\noindent\emph{if $f\in C(J^2)$ and $1\le p<\infty$, then}
\begin{equation}\label{int2}
\Vert f-L(f,{\rm LC}_{\n},\cdot)\Vert_{L_{p,w}(J^2)}\to 0\quad\text{\emph{as}}\quad\min\{n_1,n_2\}\to \infty.
\end{equation}

\smallskip

Results of qualitative type were considered only in the case of the Padua points (see~\cite{BDeMVXu2007}):

\smallskip

\noindent\emph{if $f\in C(J^2)$, then}
\begin{equation*}
  \Vert f-L(f,{\rm Pad}_{n},\cdot)\Vert_{L_{p,w}(J^2)}\le C_p E(f,\mathcal{P}_n^2)_\infty,
\end{equation*}
where $E(f,\mathcal{P}_n^2)_\infty$ is the error of the best approximation of $f$ by polynomials from $\mathcal{P}_n^2$ (bivariate algebraic polynomials of degree $n$ in two variables)  in the uniform metric.

Note that analogues of~\eqref{int1} and~\eqref{int2} for the Xu points, the Padua points,  and the Lissajous-Chebyshev node points of the non-degenerate Lissajous curve were derived in~\cite{BDeMVXu2006}, \cite{BDeMV2006}, \cite{Xu1996}, and~\cite{ErbKaethnerDenckerAhlborg2015}.

In this paper, we obtain estimates of the error $\Vert f-L(f,{\rm LC}_{\n},\cdot)\Vert_{L_{p,w}(J^2)}$ for wide classes of functions including non-smooth functions of bounded variation. In particular, we show that if a bivariate function $f$ has bounded variation in the sense of Hardy-Krause (see the definition in Section~4), then
\begin{equation}\label{int4}
\Vert f-L(f,{\rm LC}_{\n},\cdot)\Vert_{L_{p,w}(J^2)}=\mathcal{O}\(\frac1{{n_1}^{1/p}}+\frac1{{n_2}^{1/p}}\),
\end{equation}
see Corollary~\ref{corL1}.

A main idea of the proof of~\eqref{int4} is that using Marcinkiewicz-Zygmund-type inequalities and the fact that ${\rm LC}_{\n}$ is a subset of a certain set of points that forms the tensor product Chebyshev grid (see Figure~1), we can estimate the $L_p$-norm of the difference $f-L(f,{\rm LC}_{\n},\cdot)$ via
the $L_p$-norm of the error of approximation of $f$ by some tensor product interpolation polynomials (see Lemma~\ref{pass}). Then we can apply some known results for classical interpolation polynomials, e.g. \cite{Pr84}, \cite{PrT}, or  \cite{PrXu}.

The paper is organized as follows: In Section~2 we consider an abstract interpolation problem and obtain two types of estimates for the error of approximation of functions by interpolation polynomials. In Section~3 we give  preliminary results and some notations related to the polynomial interpolation on the Lissajous-Chebyshev nodes. In Section~4 we formulate the main results. In Section~5 we prove auxiliary and main results of the paper.

\section{General estimates of the error of approximation by interpolation polynomials}

Let $B_p(D)$, $1\le p<\infty$, be a subspace of $B(D)$ equipped with the norm  $\Vert f\Vert_p=\Vert f\Vert_{B_p(D)}$.  Everywhere below we suppose that the set of polynomials $\Pi_\n$ is a subset of $B_p(D)$ for any $1\le p<\infty$.

Let us consider generalized Marcinkiewicz-Zygmund (MZ) inequalities associated with a norm $\Vert \cdot\Vert_p$,
node points $X_\n$, weights $\La_\n=\{\l_\k\}_{\k\in \I_\n}\subset \R_+^d$, and a finite dimensional space of polynomials~$\Pi_\n$.

We distinguish the so-called left-hand side and right-hand side MZ inequalities.
In what follows, we say that the right-hand side MZ inequality associated with $(X_\n, \La_\n, \Pi_\n,\Vert \cdot\Vert_p)$ holds with a constant $M>0$ if
\begin{equation}\label{MZb}
  M\(\sum_{\k\in \I_\n} \l_\k |P(\x_\k)|^p\)^\frac1p\le \Vert P\Vert_p\quad\text{for all}\quad P\in \Pi_\n.
\end{equation}
By analogy, we say that the left-hand side MZ inequality associated with $(X_\n, \La_\n, \Pi_\n,\Vert \cdot\Vert_p)$ holds with a constant $M>0$ if
\begin{equation}\label{MZa}
\Vert P\Vert_p \le M\(\sum_{\k\in \I_\n} \l_\k |P(\x_\k)|^p\)^\frac1p \quad\text{for all}\quad P\in \Pi_\n.
\end{equation}

Below, the error of the best approximation of $f$ by elements of $\Pi_\n$ in the norm $\Vert \cdot\Vert_p$ is given by
$$
E(f,\Pi_\n)_p=\inf_{P\in \Pi_\n}\Vert f-P\Vert_p.
$$

Let us consider two arbitrary  polynomial operators $\L_\n : B_p(D)\mapsto \Pi_\n$ and $\L_\n' : B_p(D)\mapsto \Pi_\n'\subset B_p(D)$ of "interpolation type" such that for any $f\in B_p(D)$ one has
\begin{equation}\label{MainInt}
  \L_\n f(\x_\k)=\L_\n' f(\x_\k),\quad \x_\k\in X_\n.
\end{equation}
In the following theorem, using the above MZ inequalities, we obtain an estimate of $\Vert f-\L_\n f\Vert_p$ via the sum of $E(f,\Pi_\n)_p$ and the error of approximation of $f$ by
${\L}_\n' f$.

\begin{theorem}\label{th1}
  Let $f\in B_p(D)$, $1\le p< \infty$, $X_\n\subset {X}_\n'\subset D$, $\Pi_\n\subset {\Pi}_\n'$,
and let $\La_\n=\{\l_\k\}_{\k\in \I_\n}$ and ${\La}_\n'=\{{\l}_\k'\}_{\k\in {\I}_\n'}$ be such that
for some $\g>0$ one has $\l_\k\le \g{\l}_\k'$ for all $\k\in \I_\n$. Suppose also that equalities~\eqref{MainInt} hold.  If the left-hand side  MZ inequality associated with $(X_\n, \La_\n, \Pi_\n,\Vert \cdot\Vert_p)$ holds with a constant $M>0$ and the right-hand side MZ inequality associated with $(X_\n', \La_\n', \Pi_\n',\Vert \cdot\Vert_p)$ holds with a constant $M'>0$, then
\begin{equation}\label{eq4}
  \Vert f-\L_\n f\Vert_p\le \(\frac{\g  M}{M'}+1\) \(E(f,\Pi_\n)_p+\Vert f-{\L}_\n' f\Vert_p\).
\end{equation}
\end{theorem}

\begin{proof}
Let $P\in \Pi_\n$ be such that $\Vert f-P\Vert_p=E(f,\Pi_{\n})_p$. We have
\begin{equation}\label{eq5}
  \Vert f-\L_\n f\Vert_p\le \Vert f-P\Vert_p+\Vert P-\L_\n f\Vert_p.
\end{equation}
By~\eqref{MZa}, \eqref{MainInt}, and~\eqref{MZb}, we obtain
\begin{equation}\label{eq6}
  \begin{split}
    \Vert P-\L_\n f\Vert_p&\le M\(\sum_{\k\in \I_\n}\l_\k |P(\x_\k)-\L_\n f(\x_\k)|^p\)^\frac1p\\
&=M\(\sum_{\k\in \I_\n}\l_\k |P(\x_\k)-{\L}_\n' f(\x_\k)|^p\)^\frac1p\\
&\le \g M\(\sum_{\k\in {\I}_\n'}{\l}_\k' |P({\x}_\k')-{\L}_\n' f({\x}_\k')|^p\)^\frac1p\\
&\le \frac{\g  M}{M'} \Vert P-{\L}_\n' f\Vert_p\\
&\le \frac{\g  M}{M'}\(\Vert f-P\Vert_p+\Vert f-{\L}_\n' f\Vert_p\).
  \end{split}
\end{equation}

Finally, combining~\eqref{eq5} and~\eqref{eq6}, we get~\eqref{eq4}.
\end{proof}

In the next theorem, we estimate the error of approximation of a function $f$ by a particular interpolation process  $\L_\n f$ by means of the error of the best one-sided approximation given by
$$
\widetilde{E}(f,\Pi_{\n})_{p}=\inf\left\{\Vert Q_{\n}-q_{\n}\Vert_{p}\,:\, q_{\n}, Q_{\n} \in \Pi_{\n},\, q_{\n}(x)\le f(x)\le Q_{\n}(x),\, x\in D\right\}.
$$

\begin{theorem}\label{th2}
  Let $f\in B_p(D)$, $1\le p< \infty$, and $\L_\n f(\x)=L(f,X_\n,\x)$. If the right-hand side and the left-hand side MZ inequalities associated with $(X_\n, \La_\n, \Pi_\n,\Vert \cdot\Vert_p)$ hold with constants $M_1$ and $M_2$, correspondingly, then
\begin{equation}\label{eq7}
  \Vert f-\L_\n f\Vert_p\le \(\frac{M_2}{M_1}+1\) \widetilde{E}(f,\Pi_\n)_p.
\end{equation}
\end{theorem}

\begin{proof}
Let $q_{\n}, Q_{\n} \in \Pi_{\n}$ be such that $q_{\n}(x)\le f(x)\le Q_{\n}(x)$ and $\Vert Q_{\n} - q_{\n}\Vert_p = \widetilde{E}(f,\Pi_{\n})_p$. We have
\begin{equation}\label{eq8}
  \Vert f-\L_\n f\Vert_p\le \Vert q_\n -f\Vert_p+\Vert q_\n-\L_\n f\Vert_p.
\end{equation}
It is obvious that
\begin{equation}\label{eq9}
  \Vert q_\n -f\Vert_p\le \Vert q_\n -Q_\n\Vert_p= \widetilde{E}(f,\Pi_\n)_p.
\end{equation}
Let us estimate the second term in the right-hand side of~\eqref{eq8}. By~\eqref{MZb} and~\eqref{MZa}, we derive
\begin{equation}\label{eq10}
  \begin{split}
    \Vert q_\n-\L_\n f\Vert_p&\le M_2\(\sum_{\k\in \I_\n}\l_\k|q_\n(\x_\k)-\L_\n f(\x_\k)|^p\)^\frac1p\\
    &= M_2\(\sum_{\k\in \I_\n}\l_\k|q_\n(\x_\k)-f(\x_\k)|^p\)^\frac1p\\
    &\le M_2\(\sum_{\k\in \I_\n}\l_\k|q_\n(\x_\k)-Q_\n(\x_\k)|^p\)^\frac1p\\
    &\le \frac{M_2}{M_1}\Vert q_\n-Q_\n\Vert_p= \frac{M_2}{M_1} \widetilde{E}(f,\Pi_\n)_p.
  \end{split}
\end{equation}
Thus, combining~\eqref{eq8}--\eqref{eq10}, we get~\eqref{eq7}.
\end{proof}

\section{Preliminary results and some notations}

In this section, we give some notations and preliminary results related to the bivariate interpolation on the Lissajous-Chebyshev node points.

Let $\Np^2=\{(m,n)\in \N^2\,:\,  m\,\, \text{and}\,\, n\,\, \text{are relatively prime}\}$.
In what follows, for simplicity, we use the notation
\begin{equation}\label{xy}
  x_k=\cos\(\frac{k\pi}{m}\),\quad y_l=\cos\(\frac{l\pi}{n}\), \quad k=0,\dots,m, \quad l=0,\dots,n,
\end{equation}
to abbreviate the Chebyshev-Gauss-Lobatto points.

Recall that for each $(m,n)\in \Np^2$, the Lissajous curve $\g_{m,n}:\R\to J^2$ is given by
$$
\g_{m,n}(t)=\(\cos(nt), \cos(mt)\)
$$
and the Lissajous-Chebyshev node points of the (degenerate) Lissajous curve are defined by
$$
{\rm LC}_{m,n}=\left\{\g_{m,n}\(\frac{\pi k}{mn}\),\,k=0,\dots,mn\right\}
$$
(see Figure 1).

\begin{figure}[ht]
\includegraphics[scale=0.4, clip]{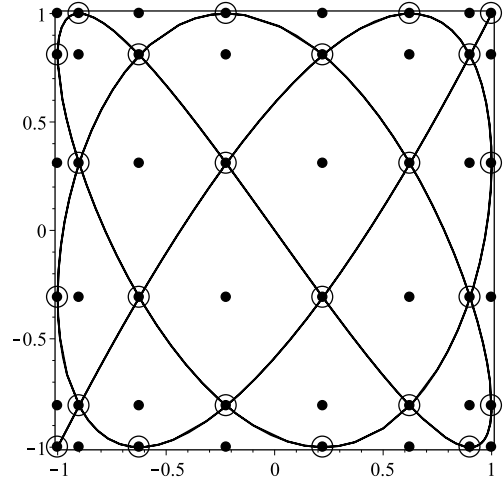}
	\caption{Lissajous curve $\gamma_{7,5}$; points of ${\rm LC}_{7,5}$ (big circles); points of the tensor product Chebyshev grid $\left\{(x_i,y_j): i=0,\dots,7,\, j=0,\dots,5 \right\}$ (small solid circles).}
\end{figure}

\bigskip

\begin{proposition}\label{prop_int} { (See~\cite{E}).}
  Let $({m,n})\in \Np^2$.
The set ${\rm LC}_{m,n}$ contains  $(m+1)(n+1)/2$ distinct points and can be represented in the following form
  $$
  {\rm LC}_{m,n}=\left\{(x_i,y_j): (i,j)\in \I_{m,n} \right\},
  $$
where
   $$
  \I_{m,n}=\left\{(i,j)\in \Z_+^2: i=0,\dots,m,\, j=0,\dots,n,\, i+j=0\pmod 2\right\}.
  $$
\end{proposition}

Further, to study the interpolation problem on ${\rm LC}_{m,n}$, we need to introduce some additional notations.
Let
$$
C_n(x)=\cos(n \arccos x)
$$
be the Chebyshev polynomial of the first kind. The normalized Chebyshev polynomial is defined by
$$
\hat{C}_n(x)=\left\{
               \begin{array}{ll}
                 1, & \hbox{$n=0$,} \\
                 \sqrt{2}C_n(x), & \hbox{$n\neq 0$.}
               \end{array}
             \right.
$$

A proper set of polynomials for interpolation on ${\rm LC}_{m,n}$ is given by
\begin{equation*}
  \Pi_{m,n}^\mathcal{L}={\rm span}\{C_{k}(x)C_{l}(y)\,:\, (k,l)\in \Gamma_{m,n}^\mathcal{L}\},
\end{equation*}
where
$$
\G_{m,n}^\L=\left\{(i,j)\in \Z_+^2 : \frac{i}{m}+\frac{j}{n}<1\right\} \cup \{(0,n)\}.
$$

To construct the Lagrange interpolation polynomial corresponding to the sets ${\rm LC}_{m,n}$ and  $\Pi_{m,n}^\mathcal{L}$, we will use the following weights
\begin{equation}\label{lll}
  \l_{i,j}:=\left\{
        \begin{array}{ll}
          \displaystyle1/{(2mn)}, & \hbox{$(x_i,y_j)$ is a vertex point of $J^2$,} \\
          \displaystyle1/(mn), & \hbox{$(x_i,y_j)$ is an edge point of $J^2$,} \\
          \displaystyle2/(mn), & \hbox{$(x_i,y_j)$ is an interior point of $J^2$.}
        \end{array}
      \right.
\end{equation}

Let us present two basic theorems for our further consideration.

\begin{theorem}\label{thErb} (See~\cite{E}).
For $f\in B(J^2)$, the unique solution of the interpolation problem
\begin{equation*}
  \L_{m,n}(f)(x_k,y_l)=f(x_k,y_l),\quad  (x_k,y_l)\in {\rm LC}_{m,n},
\end{equation*}
in the space $\Pi_{n,m}^\mathcal{L}$ is given by the polynomial
\begin{equation}\label{pol}
  \L_{m,n}f(x,y)=\sum_{(k,l)\in \I_{m,n}}f(x_k,y_l)\ell_{m,n}(x,y;x_k,y_l),
\end{equation}
where
  \begin{equation}\label{E1}
  \begin{split}
    \ell_{m,n}(x,y;x_k,y_l)&=\l_{k,l}\(\sum_{(i,j)\in \Gamma_{m,n}^\mathcal{L}} \hat{C}_{i}(x_k)\hat{C}_{j}(y_l)\hat{C}_{i}(x)\hat{C}_{j}(y)-\frac12 \hat{C}_{n}(y_l)\hat{C}_{n}(y)\).
  \end{split}
  \end{equation}
These polynomials form the fundamental polynomials of Lagrange interpolation in the space $\Pi_{m,n}^\mathcal{L}$ with respect to the point set ${\rm LC}_{m,n}$, i.e.
  $
  \ell_{m,n}(x_{k'},y_{l'};x_k,y_l)=\delta_{(k,l),(k',l')}.
  $
\end{theorem}

The next theorem was obtained in~\cite{DEKL}. In what follows for $r,s\in \Z_+$, we denote
$$
f^{(r,s)}(x,y)=\frac{\partial^{r+s}f}{\partial x^{r}\partial y^{s}}(x,y).
$$
\begin{theorem}\label{thDEKL}
  For $f\in C(J^2)$ and $(m,n)\in \Np^2$, we have
\begin{equation*}
  \Vert f-\L_{m,n}f\Vert_{L_\infty(J^2)} \le  C\log(m+1)\log(n+1)E(f,\Pi_{m,n}^\mathcal{L})_{L_\infty(J^2)},
\end{equation*}
where $C$ is some absolute constant. In particular, if $f^{(r,0)}, f^{(0,s)}\in C(J^2)$ for some $r,s\in \Z_+$, then
\begin{equation*}
\begin{split}
    \Vert f-&\L_{m,n}(f)\Vert_{L_\infty(J^2)} \\
&\le  C_{r,s}\log(m+1)\log(n+1)\({m^{-r}}\w^{(1)}(f^{(r,0)},{m^{-1}})_\infty+{n^{-s}}\w^{(2)}(f^{(0,s)},{n^{-1}})_\infty\),
\end{split}
\end{equation*}
where
$$
\w^{(1)}(f,\d)_\infty=\sup_{|h|\le 1-\d}\Vert f(\cdot+h,\cdot)-f\Vert_\infty,
$$
$$
\w^{(2)}(f,\d)_\infty=\sup_{|h|\le 1-\d}\Vert f(\cdot,\cdot+h)-f\Vert_\infty
$$
are the partial moduli of continuity by the first and second variables, correspondingly.
\end{theorem}

\section{Main results}

Our goal is to obtain analogues of Theorem~\ref{thDEKL} in the weighted $L_p$-spaces.
We say that a real bivariate function $f:D\to \R^2$ belongs to the weighted space $L_{p,u}(D)$, $1\le p<\infty$, with the weight $u(x,y)\ge 0$  if
$$
\Vert f\Vert_{L_{p,u}(D)}=\(\int_{D}|f(x,y)u(x,y)|^p{\rm d}x{\rm d}y\)<\infty.
$$
We denote the Chebyshev-type weight $w$ in the case $D=J^2=[-1,1]^2$  by
$$
w(x,y)=\frac1{\(\sqrt{1-x^2}\sqrt{1-y^2}\)^\frac1p}.
$$
In the one-dimensional case, the function $w$ has the same meaning, i.e. $w(x)=(1-x^2)^{-1/(2p)}$. In what follows, if $u\equiv 1$, then we denote
$\Vert f\Vert_{L_{p}(D)}=\Vert f\Vert_{L_{p,u}(D)}$.

Everywhere below, $c_p$ and $C_p$ denote some positive constants depending only on $p$.

From Theorem~\ref{th2} and Lemma~\ref{lemmz} in Section~5 below, we directly obtain the following quite standard result.

\begin{theorem}\label{thos}
  Let $f\in B(J^2)$, $1\le p<\infty$, and $(m,n)\in \Np^2$. Then
\begin{equation*}
  \Vert f-\L_{m,n}f\Vert_{L_{p,w}(J^2)}\le C_pK_p({m,n})\widetilde{E}(f,\Pi_{m,n}^\mathcal{L})_{L_{p,w}(J^2)},
\end{equation*}
where
\begin{equation}\label{K_p}
  K_p(m,n)=\left\{
            \begin{array}{ll}
              \log(m+1)\log(n+1), & \hbox{$p=1$,}\\
              1, & \hbox{$1<p<\infty$.}
            \end{array}
          \right.
\end{equation}
\end{theorem}

The error of the best one-sided approximation $\widetilde{E}(f,\Pi_{m,n}^\mathcal{L})_{L_{p,w}(J^2)}$ is a rather specific quantity, for which, in fact, it is difficult to find effective estimates via more classical objects. In the case of standard polynomial spaces,  some estimates of the error of the best one-sided approximation can be found in~\cite{SP}, \cite{HI}, \cite{H}.
For example, it is well-known that in the one-dimensional case for any periodic function $f$ of bounded variation on $\T\simeq [-\pi,\pi)$ and $1\le p<\infty$, one has
(see Lemma~\ref{lem2v})
\begin{equation}\label{bv1}
  \widetilde{E}(f,\Tp_n)_{L_{p}(\T)}=\mathcal{O}({n^{-1/p}}),
\end{equation}
where $\Tp_n$ is the set of all univariate real trigonometric polynomials of order at most $n$. Note that~\eqref{bv1} implies some well-known estimates of the error of approximation of $f$ by the classical Lagrange interpolation polynomials on equidistant nodes~\cite{PrXu} (see also~\cite{Pr84}, \cite{Pr89}).

We do not know if it is possible to obtain  similar estimates in the multivariate case which can be used together with Theorem~\ref{thos} to obtain efficient estimates of $\Vert f-\L_{m,n}f\Vert_{L_{p,w}(J^2)}$. In this paper, instead of utilizing Theorem~\ref{thos}, we apply another approach based on Theorem~\ref{th1}.
To introduce this approach, we need to recall some notation.

Let $BV(I)$, $I=[a,b]$, be the set of all real-valued functions $f: I\to \R$ with bounded variation on $I$.
Recall that $f\in BV(I)$ if
$$
V_I (f)=\sup_{a\le\xi_1<\dots<\xi_n\le b}\sum_{k=1}^{n-1}|f(\xi_{k+1})-f(\xi_k)|<\infty.
$$

A real-valued bivariate function $f\,:\,I^2 \to \R$ is said to be of bounded variation on $I^2$ in the sense of Hardy-Krause (see~\cite{CA})
if $f(\cdot,y), f(x,\cdot) \in BV(I)$ for some fixed $x,y\in I$ and if
$$
H_{I^2} (f)=\sup_{\pi_1, \pi_2}\sum_{k=1}^{m-1}\sum_{l=1}^{n-1}|\D f(\xi_k,\eta_l)|<\infty,
$$
where
$$
\D f(\xi_k,\eta_l)=f(\xi_{k+1},\eta_{l+1})-f(\xi_{k+1},\eta_{l})-f(\xi_{k},\eta_{l+1})+f(\xi_{k},\eta_{l})
$$
and
$$
\pi_1: a\le\xi_1<\dots<\xi_m\le b,
$$
$$
\pi_2: a\le\eta_1<\dots<\eta_n\le b
$$
are arbitrary decompositions of $I$.
The class of such functions $f$ we denote by $HBV(I^2)$. Other classes of functions of multivariate bounded variation as well as their relation to the class $HBV(I^2)$ can be found, e.g., in~\cite{CA}.

Note that if $f\in HBV(I^2)$, then $f(\cdot,y)$, $f(x,\cdot)\in BV(I)$ for any fixed $x,y\in I$. In what follows the values
$$
V_{1,I} (f)(y)\quad\text{and}\quad V_{2,I} (f)(x)
$$
are defined as the total variations of $f(\cdot,y)$ and $f(x,\cdot)$ on $I$. Note that for $f\in HBV(I^2)$ we also have $V_{1,I} (f)$, $V_{2,I} (f)\in BV(I)$ (see~\cite{CA}). Hence, $\sup_{y\in I}V_{1,I} (f)(y)$ and $\sup_{x\in I}V_{2,I} (f)(x)$ are finite.

Denote
$$
\widetilde{D}^{(r,s)}f(x,y)=\left\{\frac{\partial^{r+s}}{\partial^r\!\phi \,\partial^s\psi}f(\cos\phi,\cos\psi)\right\}_{\phi=\arccos x,\, \psi=\arccos y}.
$$
The following theorem is one of the main results of this paper.

\begin{theorem}\label{thL1}
Let $f\in B(J^2)$, $1\le p<\infty$, $r,s\in \Z_+$, and $(m,n)\in \Np^2$. If  $\widetilde{D}^{(r,s)}f\in HBV(J^2)$, then
\begin{equation*}\label{est1}
\begin{split}
    \Vert f-\L_{m,n}f&\Vert_{L_{p,w}(J^2)}\le CK_p(m,n)\bigg( {m^{-\frac1p-r}}\Vert V_{1,J} (\widetilde{D}^{(r,0)}f)\Vert_{L_{p,w}(J)}\\
&+{n^{-\frac1p-s}}\Vert V_{2,J} (\widetilde{D}^{(0,s)}f)\Vert_{L_{p,w}(J)}+m^{-\frac1p-r}n^{-\frac1p-s}H_{J^2}(\widetilde{D}^{(r,s)}f)\bigg),
\end{split}
\end{equation*}
where the constant $C$ does not depend on $m$, $n$, and $f$.
\end{theorem}

Denote $\vp(t)=\sqrt{1-t^2}$,  $\vp_1(x,y)=\vp(x)$, and $\vp_2(x,y)=\vp(y)$.

\begin{corollary}\label{corL1}
Let $f\in B(J^2)$, $1\le p<\infty$, $r,s\in \{0,1\}$, and $(m,n)\in \Np^2$. If $\vp_{1}^r\vp_{2}^sf^{(r,s)}\in HBV(J^2)$, then
\begin{equation*}
\begin{split}
    \Vert f-\L_{m,n}f&\Vert_{L_{p,w}(J^2)}\le CK_p(m,n)\bigg( {m^{-\frac1p-r}}\Vert V_{1,J} (\vp_1^rf^{(r,0)})\Vert_{L_{p,w}(J)}\\
&+{n^{-\frac1p-s}}\Vert V_{2,J} (\vp_2^sf^{(0,s)})\Vert_{L_{p,w}(J)}+m^{-\frac1p-r}n^{-\frac1p-s}H_{J^2}(\vp_1^r\vp_2^sf^{(r,s)})\bigg),
\end{split}
\end{equation*}
where the constant $C$ does not depend on $m$, $n$, and $f$.
\end{corollary}

Sharper results can be obtained in terms of moduli of smoothness and the errors of the best approximation of functions by algebraic polynomials. To formulate such results, let us recall some notations.

Denote
$$
\Delta_h^\nu (f,x)=\left\{
                     \begin{array}{ll}
                       \displaystyle{\sum_{i=1}^\nu}\binom{\nu}{i}(-1)^{\nu-i}f(x+(i-\nu/2)h), & \hbox{$x\pm \nu h/2\in J$,} \\
                       \displaystyle0, & \hbox{otherwise.}
                     \end{array}
                   \right.
$$
The weighted Ditzian-Totik modulus of smoothness of $f\in L_{p,u}(J)$ of order $\nu\in \N$ and step $t>0$ is defined by
\begin{equation}\label{DTmod}
\begin{split}
  \w_\nu^\vp(f,t)_{p,u}=\sup_{0<h\le t}&\Vert u \D_{h\vp}^\nu f\Vert_{L_p[-1+t^*,1-t^*]}\\
&+\sup_{0<h\le t}\Vert u \overrightarrow{\D}_{h}^\nu f\Vert_{L_p[-1,-1+At^*]}+\sup_{0<h\le t}\Vert u \overleftarrow{\D}_{h}^\nu f\Vert_{L_p[1-At^*,1]},
\end{split}
\end{equation}
where $t^*=2\nu^2t^2$, $A$ is an absolute constant, and the forward and backward $\nu$th differences are given by
$$
\overrightarrow{\D}_h^\nu f(x)=\D_h^\nu (f,x+\nu h/2),\quad \overleftarrow{\D}_h^\nu f(x)=\D_h^\nu (f,x-\nu h/2)
$$
(for details see~\cite[p.~218]{DT}).

In our work, we mainly deal with the following analogue of~\eqref{DTmod} given by
\begin{equation}\label{newmod}
  \w_{\nu,\a}^\vp(f,t)_p=\sup_{0\le h\le t} \Vert \mathcal{W}_{\nu h}^\a \D_{h\vp}^\nu (f,\cdot)\Vert_{L_p(J)},
\end{equation}
where
$$
\mathcal{W}_\d(x)=\(\(1-x-\frac{\d\vp(x)}{2}\)\(1+x-\frac{\d\vp(x)}{2}\)\)^\frac12.
$$
The modulus~\eqref{newmod} has been recently introduced in~\cite{KLS14}. It is known (see~\cite{KLS}) that at least for $\a\in \N$ the modulus~\eqref{newmod}  is equivalent to the weighted Ditzian-Totik modulus of smoothness~\eqref{DTmod} with $u=\vp^\a$.

Let us extend the modulus of smoothness~\eqref{newmod} to the bivariate case. For this we denote $\mathcal{W}_\d^1(x,y)=\mathcal{W}_\d(x)$ and $\mathcal{W}_\d^2(x,y)=\mathcal{W}_\d(y)$. Let also $w_1(x,y)=w(x)$ and $w_2(x,y)=w(y)$.

For an admissible $f\,:\, J^2\to \R$, the partial moduli of smoothness  related to~\eqref{newmod} are defined by
\begin{equation*}\label{newmodP}
  \w_{\nu,\a}^{\vp,j}(f,t)_p=\sup_{0\le h\le t} \Vert \mathcal{W}_{\nu h}^{\a,j} \D_{h\vp_j}^{\nu,j} (f,\cdot)\Vert_{L_{p,w_i}(J^2)},\quad i,j=1,2,\quad i\neq j,
\end{equation*}
where $\D_{h\vp_j}^{\nu,j}=\D_{h\vp_je_j}^{\nu}$ and $e_1=(1,0)$, $e_2=(0,1)$.
The corresponding mixed modulus of smoothness is given by
\begin{equation*}\label{newmodM}
  \w_{\nu,\a_1,\a_2}^{\vp}(f,t_1,t_2)_p=\sup_{0\le h_j\le t_j,\,j=1,2} \Vert \mathcal{W}_{\nu h_1}^{\a_1,1}\mathcal{W}_{\nu h_2}^{\a_2,2} \D_{h_1\vp_1}^{\nu,1}\D_{h_2\vp_2}^{\nu,2} (f,\cdot)\Vert_{L_{p}(J^2)}.
\end{equation*}

Now we are ready to formulate a result which gives a refinement of Theorem~\ref{thL1} in the case of smooth functions.
In what follows, we denote the set of all (locally) absolutely continuous functions on $[a,b]$ by $AC([a,b])$ ($AC_{\rm loc}([a,b])$).

\begin{theorem}\label{thL2}
 Let $f\in B(J^2)$, $1\le p<\infty$, $r,s\in\N$, $\nu\in \N$, and $(m,n)\in \Np^2$. Suppose $f^{(r-1,0)}(\cdot,y)$ and $f^{(r,s-1)}(\cdot,y)$ or $f^{(0,s-1)}(x,\cdot)$ and $f^{(r-1,s)}(x,\cdot)$ belong to  $AC_{\rm loc}(J)$ for a.e. $x,y\in J$.  If, in addition, $\vp_1^rf^{(r,0)}$, $\vp_2^sf^{(0,s)}$, $\vp_{1}^r\vp_{2}^sf^{(r,s)}\in L_{p,w}(J^2)$, and $m>r$, $n>s$, then
\begin{equation*}\label{trig}
\begin{split}
    \Vert f-&\L_{m,n} f\Vert_{L_{p,w}(J^2)}\le CK_p(m,n)\Bigg( {m^{-r}}\w_{\nu,r-1/p}^{\vp,1}\(f^{(r,0)},m^{-1}\)_{p}
    \\
    &+{n^{-s}}\w_{\nu,s-1/p}^{\vp,2}\(f^{(0,s)},n^{-1}\)_{p}
+{m^{-r}}n^{-s}\w_{\nu,r-1/p,s-1/p}^{\vp}\(f^{(r,s)},m^{-1},n^{-1}\)_{p}\Bigg),
\end{split}
\end{equation*}
where the constant $C$ does not depend on $m$, $n$, and $f$.
\end{theorem}

\begin{remark}\label{rem0}
Note that Theorem~\ref{thL1} is valid also for discontinuous functions of bounded variation while in Theorem~\ref{thL2} we have to require the existence of partial and mixed derivatives of at least first order. We have a similar situation in Theorem~\ref{thL+} and Theorem~\ref{thL++} stated below.
\end{remark}


It is also possible to obtain estimates of the error of approximation of functions $f$ by interpolation polynomials $\L_{m,n}f$ without using mixed moduli of smoothness, mixed derivatives, and the notion of bounded variation in the sense of Hardy-Krause.
To formulate such results we need some additional notation.

Let $\mathcal{P}_n$ be the space of all real-valued univariate algebraic polynomials of degree at most~$n$. In what follows, for simplicity, we denote
$$
E_n^\mathcal{P}(f)_{L_{p,u}(J)}=E(f,\mathcal{P}_n)_{L_{p,u}(J)}.
$$
Let also
$$
E_{m,\infty}^\Pp(f)_{L_{p,u}(J^2)}=E(f,\Pp_{m,\infty}(L_{p,u}))_{L_{p,u}(J^2)},
$$
where $\Pp_{m,\infty}(L_{p,u})$ is a class of functions $g$ such that $g\in L_{p,u}(J^2)$ and $g$ is an algebraic polynomial of degree at most $m$ in the first variable. By analogy, we define the class $\Pp_{\infty,n}(L_{p,u})$ and
$$
E_{\infty,n}^\Pp(f)_{L_{p,u}(J^2)}=E(f,\Pp_{\infty,n}(L_{p,u}))_{L_{p,u}(J^2)}.
$$

In what follows, for a sequence $\{a_k\}_{k\in \mathcal{I}_{n}}$, $\mathcal{I}_{n}=\{0,1,\dots,n-1\}$, of numbers, we denote
$$
\Vert \{a_k\}\Vert_{\widetilde{\ell}_p^n}=\(\frac1n\sum_{k=0}^{n-1} |a_k|^p\)^\frac1p.
$$

The following result is a counterpart of Theorem~\ref{thL2} for the errors of the best approximation. Recall that the points $x_k$ and $y_l$ are given by~\eqref{xy}.

\begin{theorem}\label{thL+}
  Let $f\in B(J^2)$, $1\le p<\infty$, $r,s\in \N$, and $(m,n)\in \Np^2$.
Suppose $f^{(r-1,0)}(\cdot,y)$, $f^{(0,s-1)}(x,\cdot)\in AC_{\rm loc}(J)$ for a.e. $x,y\in J$. If, in addition, $\vp_1^rf^{(r,0)}\in L_{p,w}(J^2)$ and $\vp_2^sf^{(0,s)}(x_k,\cdot)\in L_{p,w}(J)$ for all $k\in \mathcal{I}_{6m}$ and $m>r$, $n>s$, then
\begin{equation*}\label{L+1}
\begin{split}
    \Vert f-\L_{m,n}f\Vert_{L_{p,w}(J^2)}\le CK_p(m,n)\Big(m^{-r} &E_{m-r,\infty}^\Pp(f^{(r,0)})_{L_{p,\vp_1^r w}(J^2)}\\
&+n^{-s} \Vert \{E_{n-s}^\Pp(f^{(0,s)}(x_k,\cdot))_{L_{p,\vp_2^s w}(J)}\}\Vert_{\widetilde{\ell}_p^{6m}}\Big),
\end{split}
\end{equation*}
where the constant $C$ does not depend on $m$, $n$, and $f$.
\end{theorem}

\begin{remark}\label{rem1}
  The result symmetric  to Theorem~\ref{thL+} has the following form: If we suppose in Theorem~\ref{thL+} that $\vp_2^sf^{(0,s)}\in L_{p,w}(J^2)$ and $\vp_1^sf^{(r,0)}(\cdot,y_k)\in L_{p,w}(J)$ for all $k\in \mathcal{I}_{6n}$, then
\begin{equation*}
\begin{split}
    \Vert f-\L_{m,n}f\Vert_{L_p(J^2)}\le CK_p(m,n)\Big(n^{-s} &E_{\infty,{n-s}}^\Pp(f^{(0,s)})_{L_{p,\vp_2^s w}(J^2)}\\
&+m^{-r} \Vert \{E_{m-r}^\Pp(f^{(r,0)}(\cdot,y_k))_{L_{p,\vp_1^r w}(J)}\}\Vert_{\widetilde{\ell}_p^{6n}}\Big),
\end{split}
\end{equation*}
where the constant $C$ does not depend on $m$, $n$, and $f$.
\end{remark}

For not necessarily smooth functions, we obtain the following analogues of Theorem~\ref{thL1} in terms of the classical bounded variation.

\begin{theorem}\label{thL++}
  Let $f\in B(J^2)$, $1\le p<\infty$, $r,s\in \Z_+$, and $(m,n)\in \Np^2$. Suppose that $\widetilde{D}^{(r,0)}f(\cdot,y)\in BV(J)$ for a.e. $y\in J$ and $\widetilde{D}^{(0,s)}f(x_k,\cdot)\in BV(J)$ for all $k\in \I_{6m}$. If, in addition, $V_{1,J}(\widetilde{D}^{(r,0)}f)\in L_{p,w}(J)$, then
\begin{equation*}\label{v++1}
\begin{split}
    \Vert f-&\L_{m,n}f\Vert_{L_{p,w}(J^2)}\\
&\le CK_p(m,n)\({m^{-r-1/p}} \Vert V_{1,J}(\widetilde{D}^{(r,0)}f)\Vert_{L_{p,w}(\T)}+{n^{-s-1/p}} \Vert \{V_{2,J} (\widetilde{D}^{(0,s)}f(x_k,\cdot))\}\Vert_{\widetilde{\ell}_p^{6m}}\),
\end{split}
\end{equation*}
where the constant $C$ does not depend on $m$, $n$, and $f$.
\end{theorem}

\begin{remark}\label{rem2}
  The result symmetric  to Theorem~\ref{thL++} has the following form: If we suppose in Theorem~\ref{thL++} that
$\widetilde{D}^{(0,s)}f(x,\cdot)\in BV(J)$ for a.e. $x\in J$ and $\widetilde{D}^{(r,0)}f(\cdot,y_k)\in BV(J)$ for all $k\in \I_{6n}$, and, in addition, $V_{2,J}(\widetilde{D}^{(0,s)}f)\in L_{p,w}(J)$, then
\begin{equation*}\label{v+2}
\begin{split}
    \Vert f-&\L_{m,n}f\Vert_{L_{p,w}(J^2)}\\
&\le CK_p(m,n)\({n^{-s-1/p}} \Vert V_{2,J}(\widetilde{D}^{(0,s)}f)\Vert_{L_{p,w}(\T)}+{m^{-r-1/p}} \Vert \{V_{1,J} (\widetilde{D}^{(r,0)}f(\cdot,y_l))\}\Vert_{\widetilde{\ell}_p^{6n}}\),
\end{split}
\end{equation*}
where the constant $C$ does not depend on $m$, $n$, and $f$.
\end{remark}

\section{Proof of the main results}

\subsection{Marcinkiewicz-Zygmund type inequalities}

Let us start from the MZ type inequality for algebraic polynomials  $P$ from $\Pi_{m,n}^\mathcal{L}$.
\begin{lemma}\label{lemmz}
Let $1\le p<\infty$ and $P\in \Pi_{m,n}^\mathcal{L}$. Then
\begin{equation}\label{mza1}
 c_p\(\sum_{(k,l)\in \I_{m,n}}\l_{k,l}|P(x_k,y_l)|^p\)^\frac1p \le \Vert P\Vert_{L_{p,w}(J^2)},
\end{equation}
\begin{equation}\label{mza2}
 \Vert P\Vert_{L_{p,w}(J^2)}\le C_pK_p(m,n)\(\sum_{(k,l)\in \I_{m,n}}\l_{k,l}|P(x_k,y_l)|^p\)^\frac1p,
\end{equation}
where $\lambda_{k,l}$ and $K_p(m,n)$ are given by~\eqref{lll} and~\eqref{K_p}, correspondingly.
\end{lemma}

\begin{proof}
The proof of~\eqref{mza1} for all $1\le p<\infty$ as well as  the proof of~\eqref{mza2} in the case $1<p<\infty$ one can  find in~\cite{E}. Let us consider~\eqref{mza2} in the case $p=1$. By~\eqref{pol}, we get
\begin{equation}\label{mza1x}
  \Vert P\Vert_{L_{1,w}(J^2)}\le \sum_{(k,l)\in \I_{m,n}}|P(x_k,y_l)|\Vert \ell_{m,n}(\cdot,\cdot;x_k,y_l)\Vert_{L_{1,w}(J^2)}.
\end{equation}
At the same time, from~\cite{E} (see also~\cite{DEKL}) it follows
\begin{equation}\label{mza2x}
\sup_{(x,y)\in J^2}\Vert \ell_{m,n}(\cdot,\cdot;x,y)\Vert_{L_{1,w}(J^2)}\le C(mn)^{-1}\log(m+1)\log(n+1),
\end{equation}
where $C$ is some absolute constant.

Combining \eqref{mza1x} and \eqref{mza2x}, we proved the lemma.
\end{proof}

Now let us consider a trigonometric analogue of Lemma~\ref{lemmz} for the following bivariate trigonometric polynomials given by
\begin{equation*}
  \Tp_{m,n}^{\mathcal{L}}={\rm span}\{\cos(k\phi)\cos(l\psi)\,:\, (k,l)\in \Gamma_{m,n}^\mathcal{L}\}.
\end{equation*}

\begin{lemma}\label{lemmzT}
Let $1\le p<\infty$ and $T\in \Tp_{m,n}^\mathcal{L}$. Then
\begin{equation*}
 \Vert T\Vert_{L_{p}(\T^2)}\le C_pK_p(m,n)\(\frac1{mn}\sum_{(k,l)\in \I_{m,n}}\Big|T\Big(\frac{\pi k}m,\frac{\pi l}n\Big)\Big|^p\)^\frac1p.
\end{equation*}
\end{lemma}

\begin{proof}
We only need to apply the standard trigonometric substitution $x=\cos\phi$, $y=\cos\psi$ to the inequality~\eqref{mza2} and to take into account that $T$ is an even polynomial in each variable. 
\end{proof}

For application of Theorem~\ref{th1}, we need the following polynomial set
 $$
\Tp_{m,n}'=\spann\{\cos(k\phi)\cos(l\psi)\,:\, (k,l)\in \I_{m,n}'\},
$$
where the index set is given by
 $$
  {\I}_{m,n}'=\left\{(i,j)\in \Z_+^2: \begin{array}{ccc}
                                              i=0,\dots,6m-1 \\
                                              j=0,\dots,6n-1
                                            \end{array}
  \right\}.
 $$

\begin{lemma}\label{lemmz2T}
Let $1\le p<\infty$ and $T\in \Tp_{m,n}'$. Then
\begin{equation*}
 c_p\(\frac1{36mn}\sum_{(k,l)\in \I_{m,n}'}\Big|T\Big(\frac{\pi k}{3m},\frac{\pi l}{3n}\Big)\Big|^p\)^\frac1p \le \Vert T\Vert_{L_{p}(\T^2)}.
\end{equation*}
\end{lemma}

\begin{proof}
The proof easily follows from the well-known one-dimensional result, see, e.g.,~\cite{Xu}:
\begin{equation}\label{mztrig}
  \bigg(\frac1{6n}\sum_{k=0}^{6n-1}\Big|T\Big(\frac{\pi k}{3n}\Big)\Big|^p\bigg)^\frac1p\le C_p\Vert T\Vert_{L_p(\T)}, \quad T\in \Tp_{6n},
\end{equation}
where $\Tp_n$ is the set of all real-valued univariate trigonometric polynomials of order at most~$n$.
\end{proof}

\subsection{Weighted averaged moduli of smoothness}
The following  weighted Ditzian-Totik modulus of smoothness of $f\in L_{p,u}(J)$ was considered in~\cite[see (6.1.9)]{DT}
\begin{equation}\label{wem}
  \begin{split}
    \w_\nu^{*\vp}(f,t)_{p,u}=&\(\frac1t\int_0^t\int_{-1+t^*}^{1-t^*}|u(x)\D_{\tau\vp}^\nu(f,x)|^p {\rm d}x{\rm d}\tau\)^\frac1p\\
     &+\(\frac1{t^*}\int_0^{t^*}\int_{-1}^{-1+At^*}|u(x)\overrightarrow{\D}_{\tau}^\nu(f,x)|^p {\rm d}x{\rm d}\tau\)^\frac1p\\
     &+\(\frac1{t^*}\int_0^{t^*}\int_{1-At^*}^{1}|u(x)\overleftarrow{\D}_{\tau}^\nu(f,x)|^p {\rm d}x{\rm d}\tau\)^\frac1p
  \end{split}
\end{equation}
(recall that $t^*=2\nu^2t^2$ and $A$ is some absolute constant).

For $\d>0$, we denote
\begin{equation*}
  \begin{split}
     \mathcal{D}_\d=\left\{x\in \R\,:\, 1-\frac{\d \vp(x)}{2}\ge |x|\right\}\setminus \{\pm 1\}=\left\{|x|\le\frac{4-\d^2}{4+\d^2}\right\}.
   \end{split}
\end{equation*}
In~\cite{KLS}, in connection with the modulus~\eqref{newmod}, it was introduced a  "modification"\, of the  averaged modulus of smoothness~\eqref{wem}. For $ f\in L_{p,\vp^\a}(J)$, it is defined by
\begin{equation}\label{avmW}
      \w_{\nu,\a}^{*\vp}(f,t)_p=\(\frac1{t}\int_0^{t} \Vert \mathcal{W}_{\nu \tau}^{\a}\D_{\tau\vp}^\nu (f,\cdot)\Vert_{L_{p}(\mathcal{D}_{\nu\tau})}^p {\rm d}\tau\)^\frac1p.
\end{equation}
In contrast to~\eqref{wem}, this modulus of smoothness is more convenient for the goal of this paper.

We have the following inequalities for the moduli of smoothness introduced above.

\begin{lemma}\label{mmm}
Let $\nu\in \N$, $\a>0$, and $\vp^\a f\in L_p(J)$, $1\le p<\infty$. Then
\begin{equation}\label{eqmmm}
  \w_\nu^\vp (f,t)_{p,\vp^\a}\le c\w_\nu^{*\vp} (f,t)_{p,\vp^\a}\le C\w_{\nu,\a}^{*\vp} (f,t)_{p}\le C\w_{\nu,\a}^{\vp} (f,t)_{p},
\end{equation}
where the constants $c$ and $C$  do not depend on $t$ and $f$.
\end{lemma}

\begin{proof}
  The scheme of proving the first inequality in~\eqref{eqmmm} can be found in~\cite[pp.~56--57]{DT}. The second inequality can be shown analogously to the proof of Lemma~6.1 in~\cite{KLS}. The third inequality is obvious.
\end{proof}

Now let us introduce bivariate analogues of the modulus of smoothness~\eqref{avmW}.

For $f\,:\, J^2\to \R$ and $i,j=1,2$, where $i\neq j$, we define the partial averaged modulus of smoothness by
\begin{equation*}
  \w_{\nu,\a}^{*\vp,j}(f,t)_p=\(\frac1t\int_0^t {\rm d}\tau\int_{\mathcal{D}_{\nu\tau}}{\rm d} x_j\int_{-1}^1|\mathcal{W}_{\nu \tau}^{\a}(x_j) \D_{\tau\vp(x_j)}^{\nu,j} (f,x_1,x_2)|^p\frac{{\rm d} x_i}{\sqrt{1-x_i^2}}\)^\frac1p.
\end{equation*}
A corresponding mixed averaged modulus of smoothness is given by
\begin{equation*}
\begin{split}
    \w_{\nu,\a_1,\a_1}^{*\vp}&(f,t_1,t_2)_p\\
&=\(\frac1{t_1t_2}\int_0^{t_1}\int_0^{t_2} \Vert \mathcal{W}_{\nu \tau_1}^{\a_1,1}\mathcal{W}_{\nu \tau_2}^{\a_2,2} \D_{\tau_1\vp_1}^{\nu,1}\D_{\tau_2\vp_2}^{\nu,2} (f,\cdot)\Vert_{L_{p}(\mathcal{D}_{\nu\tau_1}\times \mathcal{D}_{\nu\tau_2})}^p {\rm d}\tau_1{\rm d}\tau_2\)^\frac1p.
\end{split}
\end{equation*}

\subsection{Weighted error of the best approximation by algebraic polynomials}

We will use the following Jackson-type inequalities for $E_n^\mathcal{P}(f)_{L_{p,u}(J)}$.

\begin{lemma}\label{J}
Let $0\le r<\nu$, $\a>0$, $f^{(r-1)}$ be locally absolutely continuous in $(-1,1)$, and $\vp^{\a+r}f^{(r)}\in L_p(J)$, $1\le p<\infty$. Then
\begin{equation}\label{Je1}
  E_n^\Pp(f)_{L_{p,\vp^\a}(J)}\le Cn^{-r} \w_{\nu-r}^\vp (f^{(r)},n^{-1})_{p,\vp^{\a+r}},\quad n>r.
\end{equation}
In particular,
\begin{equation}\label{Je2}
  E_n^\Pp(f)_{L_{p,\vp^\a}(J)}\le Cn^{-r} E_{n-r}^\Pp(f^{(r)})_{L_{p,\vp^{\a+r}}(J)},\quad n>r,
\end{equation}
where the constant $C$ does not depend on $n$ and $f$.
\end{lemma}

\begin{proof}
Inequality~\eqref{Je1} can be found in~\cite{KLS14}. The proof of~\eqref{Je2} follows from~\cite[Theorem~1]{Ku}. In particular, we have
$$
E_n^\Pp(f)_{L_{p,\vp^\a}(J)}=E_n^\Pp(f-P_n)_{L_{p,\vp^\a}(J)}\le Cn^{-r}\Vert \vp^r(f-P_n)^{(r)}\Vert_{L_{p,\vp^\a}(J)},
$$
where we choose the polynomial $P\in \Pp_n$ such that
$$
\Vert \vp^{r+\a}(f^{(r)}-P_n^{(r)})\Vert_{L_p(J)}=E_{n-r}^\Pp(f^{(r)})_{L_{p,\vp^{r+\a}}(J)}.
$$
\end{proof}

\subsection{De la Vall\'ee-Poussin means}

Let us consider the de la Vall\'ee-Poussin means of Lagrange interpolation polynomials (see~\cite{Sz}). These means will  play the role of  an intermediate approximant $\mathcal{L}_\n'$ in the application of Theorem~\ref{th1}.

We start form the one-dimensional case, in which de la Vall\'ee-Poussin means are given by
\begin{equation}\label{V1}
 \V_nf(\phi)=\V_n(f,\phi)=\frac1{3n}\sum_{k=0}^{6n-1}f\(t_k\)K_n\(\phi-t_k\),\quad t_k=\frac{\pi k}{3n},
\end{equation}
where
$$
K_n(\phi)=\frac12+\sum_{k=1}^{2n}\cos k\phi+\sum_{k=2n+1}^{4n-1}\frac{4n-k}{2n}\cos k\phi.
$$

Recall some basic properties of $\V_n(f,\phi)$ (see~\cite{Sz}).

\begin{lemma}\label{lem1v}
The following assertions hold:

  1)\quad $\deg \V_n f\le 4n-1$;

  2)\quad $\V_nf\(t_k\)=f\(t_k\),\quad k=0,\dots,6n-1$;

  3)\quad $\V_nT(\phi)=T(\phi)$ for any $T\in \Tp_{2n}$.

\end{lemma}

\begin{lemma}\label{lem2vE}
Let $f\in B(\T)$, $1\le p<\infty$, and $n\in \N$. Then
\begin{equation*}
  \Vert f-\V_nf\Vert_{L_p(\T)}\le C_p\widetilde{E}(f,\Tp_{n})_{L_p(\T)}.
\end{equation*}
\end{lemma}

\begin{proof}
Let $q_n, Q_n\in \Tp_n$ be such that $q_n(t)\le f(t)\le Q_n(t)$ and $\Vert Q_n-q_n\Vert_{L_p(\T)}=\widetilde{E}(f,\Tp_{n})_{L_p(\T)}$. Similarly to~\eqref{eq8} and~\eqref{eq9}, we derive
\begin{equation}\label{+Dop1}
  \Vert f-\V_nf\Vert_{L_p(\T)}\le \widetilde{E}(f,\Tp_{n})_{L_p(\T)}+\Vert q_n-\V_nf\Vert_{L_p(\T)}.
\end{equation}
Next, using Lemma~\ref{lem1v}, item 3, the inequality
\begin{equation}\label{newnew}
  \Vert \V_n f\Vert_{L_p(\T)} \le C_p\(\frac 1n\sum_{k=0}^{6n-1} |f(t_k)|^p\)^\frac1p,
\end{equation}
which can be proved similarly to Theorem~3.2.3 in~\cite{MM}, and inequality~\eqref{mztrig}, we obtain
\begin{equation}\label{+Dop2}
  \begin{split}
    \Vert q_n-\V_nf\Vert_{L_p(\T)}&=\Vert \V_n(q_n-f)\Vert_{L_p(\T)}\le C\left(\frac 1n\sum_{k=0}^{6n-1}|q_n(t_k)-f(t_k)|^p\right)^\frac1p\\
    &\le C\left(\frac 1n\sum_{k=0}^{6n-1}|q_n(t_k)-Q_n(t_k)|^p\right)^\frac1p\\
    &\le C\Vert q_n-Q_n\Vert_{L_p(\T)}=C\widetilde{E}(f,\Tp_{n})_{L_p(\T)}.
  \end{split}
\end{equation}
Finally, combining~\eqref{+Dop1} and~\eqref{+Dop2}, we proved the lemma.
\end{proof}

\begin{lemma}\label{lem2v}
  Let $f\in B(\T)$, $1\le p<\infty$, and $n\in \N$.

1) If $f$ is an absolutely continuous function on $\T$ and $f'\in L_p(\T)$, then
\begin{equation*}
  \Vert f-\V_nf\Vert_{L_p(\T)}\le {C_p}{n^{-1}} {E}(f',\Tp_n)_{L_p(\T)}.
\end{equation*}

2) If $f\in BV(\T)$, then
\begin{equation*}
  \Vert f-\V_nf\Vert_{L_p(\T)}\le {C_p}{n^{-\frac1p}}V_\T(f).
\end{equation*}
\end{lemma}

\begin{proof}
The assertions of the lemma follow from Lemma~\ref{lem2vE} and the following two inequalities:
\begin{equation*}
  \widetilde{E}(f,\Tp_n)_{L_p(\T)}\le {C_p}{n^{-1}} {E}(f',\Tp_n)_{L_p(\T)},
\end{equation*}
\begin{equation*}
  \widetilde{E}(f,\Tp_n)_{L_p(\T)}\le {C_p}{n^{-\frac1p}}V_\T(f).
\end{equation*}
The first inequality can be found, e.g., in~\cite[Theorem 8.1]{SP}. The second one follows from~\cite[see Theorem 8.2 and the formula  (7) on p.~10]{SP}.
\end{proof}

In what follows we denote
$$
f^*(\phi)=f(\cos \phi).
$$

\begin{lemma}\label{lemVf*}
  Let $r, \nu\in \N$,  $f^{(r-1)}$ be locally absolutely continuous on $(-1,1)$, and $\vp^{r}f^{(r)}\in L_{p,w}(J)$, $1\le p<\infty$. Then
\begin{equation}\label{vf*0}
  \Vert f^* -\mathcal{V}_nf^*\Vert_{L_p(\T)}\le Cn^{-r} \w_{\nu,r-1/p}^{*\vp} (f^{(r)},n^{-1})_{p},\quad n>r,
\end{equation}
where the constant $C$ does not depend on $n$ and $f$.
\end{lemma}

\begin{proof}

A simple observation is that it is sufficient to derive~\eqref{vf*0} for functions $f$ with an absolutely continuous derivative $f^{(r-1)}$ on $[-1,1]$. Indeed, first we can prove~\eqref{vf*0} for the functions $f_\rho(x)=f(\rho x)$ and then to take limit as $\rho\to 1$ (see also~\cite[p.~260]{DeLo}).

Thus, by Lemma~\ref{lem2v}, we have
\begin{equation}\label{vf*1+}
  \Vert f^*-\mathcal{V}_nf^*\Vert_{L_p(\T)}\le C n^{-1}E((f^*)',\Tp_n)_{L_p(\T)}.
\end{equation}
Let $P_{n-1}\in \mathcal{P}_{n-1}$ be such that
$$
\Vert f'-P_{n-1}\Vert_{L_{p,w\vp}(J)}=E_{n-1}^\Pp(f')_{L_{p,w\vp}(J)}.
$$
Then, using Lemma~\ref{J} and Lemma~\ref{mmm}, we derive
\begin{equation}\label{vf*2+}
  \begin{split}
     E((f^*)',\Tp_n)_{L_p(\T)}&\le \Vert \sin\phi f'(\cos\phi)-\sin\phi P_{n-1}(\cos\phi)\Vert_{L_p(\T)}\\
&=\Vert (f'-P_{n-1})\vp\Vert_{L_{p,w}(J)}=E_{n-1}^\Pp(f')_{L_{p,w\vp}(J)}\\
&\le Cn^{-r+1}\w_{\nu}^\vp(f^{(r)},n^{-1})_{p,w\vp^{r}}\\
&\le Cn^{-r+1} \w_{\nu,r-1/p}^{*\vp} (f^{(r)},n^{-1})_{p}.
   \end{split}
\end{equation}
Combining~\eqref{vf*1+} and~\eqref{vf*2+}, we get~\eqref{vf*0}.
\end{proof}

Now let us consider the multidimensional analogue of~\eqref{V1}.
Let  $f:\R^2\to\R$ be $2\pi$-periodic in each variable. We introduce the bivariate de la Vall\'ee-Poussin type interpolation operator by
$$
\V_{m,n} f(\phi,\psi)=\V_{m,n} (f,\phi,\psi)=\frac1{9mn}\sum_{(k,l)\in {\I}_{m,n}'} f(\phi_k,\psi_l) K_m(\phi-\phi_k) K_n(\psi-\psi_l),
$$
where
$$
\phi_k=\frac{\pi k}{3m},\quad \psi_l=\frac{\pi l}{3n}, \quad k=0,\dots,6m-1,\quad l=0,\dots,6n-1.
$$

We also need the following auxiliary operators:
$$
\V_{m,\infty}f(\phi,\psi)=\V_{m,\infty}(f,\phi,\psi)=\V_m(f(\cdot,\psi),\phi),
$$
$$
\V_{\infty,n}f(\phi,\psi)=\V_{\infty,n}(f,\phi,\psi)=\V_n(f(\phi,\cdot),\psi).
$$
It is easy to see that $\V_{m,n}(f,\phi,\psi)=\V_{m,\infty}(\V_{\infty,n}f,\phi,\psi)$.

In what follows, we  pass from the algebraic case of functions on $J^2$ to the case of periodic functions on $\T^2\simeq[0,2\pi)^2$, by using the standard substitution
$$
f^*(\phi,\psi)=f(\cos\phi,\cos\psi).
$$
Everywhere below we also denote $I=[0,\pi)$.

\begin{lemma}\label{pass}
Let $f\in B(J^2)$, $1\le p<\infty$, and $(m,n)\in \Np^2$. Then
\begin{equation*}
  \begin{split}
    \Vert f-&{\L}_{m,n}f\Vert_{L_{p,w}(J^2)}\\
&\le C_pK_p(m,n)\(\Vert f^*- \V_{\lfloor\frac{m}{8}\rfloor,\lfloor\frac{n}{8}\rfloor}f^*\Vert_{L_{p}(\T^2)}+\Vert f^*- \V_{m,n}f^*\Vert_{L_{p}(\T^2)}\).
   \end{split}
\end{equation*}
\end{lemma}

\begin{proof}
Consider the interpolation polynomial
\begin{equation*}
\begin{split}
  \L_{m,n}^\mathcal{T} f^*(\phi,\psi)&=\sum_{(k,l)\in \mathcal{I}_{m,n}}f^*(\phi_{3k},\psi_{3l})\ell_{m,n}(\cos\phi,\cos\psi;x_k,y_l)\\
 &=\sum_{(k,l)\in \mathcal{I}_{m,n}}f(x_k,y_l)\ell_{m,n}(\cos\phi,\cos\psi;x_k,y_l).
\end{split}
\end{equation*}
It is clear that
$$
{\L}_{m,n}f(\cos\phi,\cos\psi)=\L_{m,n}^{\mathcal{T}}f^*(\phi,\psi)
$$
and
\begin{equation}\label{eq*}
  \Vert f-{\L}_{m,n}f\Vert_{L_{p,w}(J^2)}=\Vert f^*- \L_{m,n}^{\mathcal{T}}f^*\Vert_{L_{p}(I^2)}=\frac14\Vert f^*- \L_{m,n}^{\mathcal{T}}f^*\Vert_{L_{p}(\T^2)}.
\end{equation}
It is also easy to see that $\L_{m,n}^\mathcal{T} f^*(\phi_{3k},\psi_{3l})=f^*(\phi_{3k},\psi_{3l})=\V_{m,n}f^*(\phi_{3k},\psi_{3l})$ for any $(k,l)\in \I_{m,n}$, $\L_{m,n}^{\mathcal{T}}f^*\in \Tp_{m,n}^{\mathcal{L}}\subset \Tp_{m,n}'\ni \V_{m,n}f^*$, and $(\phi_{3k},\psi_{3l})_{(k,l)\in \I_{m,n}}\subset (\phi_k,\psi_l)_{(k,l)\in \I_{m,n}'}$. Thus, using Theorem~\ref{th1} with Lemmas~\ref{lemmzT} and~\ref{lemmz2T}, and taking into account~\eqref{eq*},
we derive
\begin{equation*}
  \begin{split}
    &\Vert f-{\L}_{m,n}f\Vert_{L_{p,w}(J^2)}=\frac14\Vert f^*-\L_{m,n}^{\mathcal{T}}f^*\Vert_{L_{p}(\T^2)}\\
&\le CK_p(m,n)\(E(f^*,\Tp_{m,n}^{\mathcal{L}})_{L_p(\T^2)}+\Vert f^*- \V_{m,n}f^*\Vert_{L_{p}(\T^2)}\)\\
&\le CK_p(m,n)\(E(f^*,\Tp_{\lfloor\frac{m}{12}],[\frac{n}{12}\rfloor}')_{L_p(\T^2)}+\Vert f^*- \V_{m,n}f^*\Vert_{L_{p}(\T^2)}\)\\
&\le CK_p(m,n)\(\Vert f^*- \V_{\lfloor\frac{m}{8}\rfloor,\lfloor\frac{n}{8}\rfloor}(f^*)\Vert_{L_{p}(\T^2)}+\Vert f^*- \V_{m,n}f^*\Vert_{L_{p}(\T^2)}\).
   \end{split}
\end{equation*}
The lemma is proved.
\end{proof}

\begin{lemma}\label{thVtrig1}
Under the conditions of Theorem~\ref{thL1}, we have for any $m,n\in \N$
\begin{equation*}\label{est1V}
\begin{split}
    \Vert f^*-\V_{m,n} f^*\Vert_{L_p(\T^2)}&\le C\bigg( {m^{-\frac1p-r}}\Vert V_{1,J} (\widetilde{D}^{(r,0)}f)\Vert_{L_{p,w}(J)}\\
&+{n^{-\frac1p-s}}\Vert V_{2,J} (\widetilde{D}^{(0,s)}f)\Vert_{L_{p,w}(J)}+m^{-\frac1p-r}n^{-\frac1p-s}H_{J^2}(\widetilde{D}^{(r,s)}f)\bigg),
\end{split}
\end{equation*}
where the constant $C$ does not depend on $m$, $n$, and $f$.
\end{lemma}

\begin{proof}
From the conditions of the lemma, it follows that $(f^*)^{(r,s)}\in HBV(I^2)$. Thus, using Lemma~\ref{lem2v}, item 2, and repeating the proofs of the main results from~\cite{PrT} for the interpolation operator  $\V_{m,n} f^*$ on $I^2$, we obtain
\begin{equation}\label{trigVV}
\begin{split}
    \Vert f^*-&\V_{m,n} f^*\Vert_{L_p(\T^2)}=4\Vert f^*-\V_{m,n} f^*\Vert_{L_p(I^2)}\\
&\le C\bigg( {m^{-\frac1p-r}}\Vert V_{1,I} ((f^*)^{(r,0)})\Vert_{L_p(I)}\\
&\quad\quad\quad\quad+{n^{-\frac1p-s}}\Vert V_{2,I} ((f^*)^{(0,s)})\Vert_{L_p(I)}+m^{-\frac1p-r}n^{-\frac1p-s}H_{I^2}((f^*)^{(r,s)})\bigg).
\end{split}
\end{equation}
Using the substitution $x=\cos \phi$ and $y=\cos\psi$, it is easy to see that
\begin{equation}\label{vf*1}
  \Vert V_{1,I} ((f^*)^{(r,0)})\Vert_{L_p(I)}=\Vert V_{1,J} (\widetilde{D}^{(r,0)}f)\Vert_{L_{p,w}(J)},
\end{equation}
\begin{equation}\label{vf*2}
\Vert V_{2,I} ((f^*)^{(0,s)})\Vert_{L_p(I)}=\Vert V_{2,J} (\widetilde{D}^{(0,s)}f)\Vert_{L_{p,w}(J)},
\end{equation}
and
$$
H_{I^2}((f^*)^{(r,s)})=H_{J^2}(\widetilde{D}^{(r,s)}f).
$$

Thus, combining the above three inequalities and~\eqref{trigVV}, we proved the lemma.
\end{proof}

A sharper result can be obtained in terms of moduli of smoothness but only for smooth functions.

\begin{lemma}\label{thVtrig2}
Under the conditions of Theorem~\ref{thL2}, we have for any $m>r$ and $n>s$
\begin{equation*}
\begin{split}
    \Vert f^*-&\mathcal{V}_{m,n} f^*\Vert_{L_{p}(\T^2)}\le C\Bigg( {m^{-r}}\w_{\nu,r-1/p}^{\vp,1}\(f^{(r,0)},m^{-1}\)_{p}
    \\
    &+{n^{-s}}\w_{\nu,s-1/p}^{\vp,2}\(f^{(0,s)},n^{-1}\)_{p}
+{m^{-r}}n^{-s}\w_{\nu,r-1/p,s-1/p}^{\vp}\(f^{(r,s)},m^{-1},n^{-1}\)_{p}\Bigg),
\end{split}
\end{equation*}
where the constant $C$ does not depend on $m$, $n$, and $f$.
\end{lemma}

\begin{proof}
It is sufficient to consider the case $f^{(0,s-1)}(x,\cdot), f^{(r-1,s)}(\cdot,y)\in AC_{\rm loc}(J)$ for a.e. $x,y\in J$.
Denote
$$
f^*-\B_{m,n}f^*=(I-\V_{m,\infty})(f^*-\V_{\infty,n}f^*).
$$
By Lemma~\ref{lemVf*}, Fubini's theorem, and Lemma~\ref{mmm}, we obtain
\begin{equation}\label{B1}
  \begin{split}
    &\Vert f^*-\B_{m,n}(f^*)\Vert_{L_p(\T^2)}^p\\
&=\int_0^{2\pi} \Vert (I-\V_{m,\infty})(f^*(\cdot,\psi)-\V_{\infty,n}f^*(\cdot,\psi))\Vert_{L_p(\T)}^p {\rm d}\psi\\
&\le Cm^{-rp}\int_0^{2\pi} \w_{\nu,r-1/p}^{*\vp}( f^{(r,0)}(\cdot,\psi)-\V_{\infty,n}(f^{(r,0)}(\cdot,\psi)),m^{-1})_p^p {\rm d}\psi\\
&=Cm^{-rp+1}\int_0^{2\pi} {\rm d}\psi\int_0^{1/m}{\rm d}\tau_1\\
&\quad\quad\times\int_{\mathcal{D}_{\nu\tau_1}} \left| \mathcal{W}_{\nu\tau_1}^{r-1/p}(x_1)\D_{\tau_1\vp(x_1)}^{\nu,1} \(f^{(r,0)}(x_1,\cos\psi)-\V_{\infty,n} (f^{(r,0)}(x_1,\cdot),\cos\psi)\)\right|^p {\rm d}x_1\\
&=Cm^{-rp+1}\int_0^{1/m}\int_{\mathcal{D}_{\nu\tau_1}} |\mathcal{W}_{\nu\tau_1}^{r-1/p}(x_1)|^p {\rm d}x_1 {\rm d}\tau_1\\
&\quad\quad\times \int_0^{2\pi}  \left|\D_{\tau_1\vp(x_1)}^{\nu,1} f^{(r,0)}(x_1,\cos\psi)-\V_{\infty,n} (\D_{\tau_1\vp(x_1)}^{\nu,1}f^{(r,0)}(x_1,\cdot),\cos\psi)\right|^p{\rm d}\psi \\
&\le Cm^{-rp+1}\int_0^{1/m}\int_{\mathcal{D}_{\nu\tau_1}} |\mathcal{W}_{\nu\tau_1}^{r-1/p}(x_1)|^p {\rm d}x_1 {\rm d}\tau_1\\
&\quad\quad\times n^{-sp+1}\int_0^{1/n}\int_{\mathcal{D}_{\nu\tau_2}}
|\mathcal{W}_{\nu\tau_2}^{s-1/p}(x_2) \D_{\tau_1\vp(x_1)}^{\nu,1}\D_{\tau_2\vp(x_2)}^{\nu,2}f^{(r,s)}(x_1,x_2)|^p {\rm d}x_2 {\rm d}\tau_2\\
&\le C m^{-rp}n^{-sp}\w_{\nu,r-1/p,s-1/p}^\vp(f^{(r,s)},n^{-1},m^{-1})_p^p.
  \end{split}
\end{equation}
By analogy, we obtain the following estimates
\begin{equation}\label{B2}
\Vert f^*-\V_{m,\infty}f^*\Vert_{L_p(\T^2)}\le Cm^{-r}\w_{\nu,r-1/p}^{\vp,1}(f^{(r,0)},m^{-1})_p,
\end{equation}
\begin{equation}\label{B3}
\Vert f^*-\V_{\infty,n}f^*\Vert_{L_p(\T^2)}\le Cn^{-s}\w_{\nu,s-1/p}^{\vp,2}(f^{(0,s)},n^{-1})_p.
\end{equation}

Finally, taking into account that
$$
f^*-\V_{m,n}f^*=f^*-\V_{m,\infty}f^*+f-\V_{\infty,n}f^*+\B_{m,n}f^*-f^*
$$
and combining~\eqref{B1}--\eqref{B3}, we proved the lemma.

\end{proof}

To obtain estimates of the error of approximation under less restrictive conditions on the function, we use another approach based on the following representation (see also~\cite{PrT})
\begin{equation}\label{rep1}
\begin{split}
  f^*(\phi,\psi)-&\V_{m,n}f^*(\phi,\psi)\\
&=f^*(\phi,\psi)-\V_{m,\infty}f^*(\phi,\psi)+\V_{m,\infty}\(f^*(\cdot,\psi)-\V_{\infty,n}f^*(\cdot,\psi)\)(\phi).
\end{split}
\end{equation}

\begin{lemma}\label{thV+}
Under the conditions of Theorem~\ref{thL+}, we have for any $m>r$, $n>s$
\begin{equation}\label{v+1}
\begin{split}
    \Vert f^*-\V_{m,n}f^*\Vert_{L_p(\T^2)}\le C\Big(m^{-r} &E_{m-r,\infty}^\Pp(f^{(r,0)})_{L_{p,\vp_1^r w}(J^2)}\\
&+n^{-s} \Vert \{E_{n-s}^\Pp(f^{(0,s)}(x_k,\cdot))_{L_{p,\vp_2^s w}(J)}\}\Vert_{\widetilde{\ell}_p^{6m}}\Big),
\end{split}
\end{equation}
where the constant $C$ does not depend on $m$, $n$, and $f$.
\end{lemma}

\begin{proof}
Analogously to the proof of Lemma~\ref{lemVf*}, it is sufficient to consider only the case $f^{(r-1,0)}(\cdot,y)$, $f^{(0,s-1)}(x,\cdot)\in AC([-1,1])$ for a.e. $x,y\in J$.

By~\eqref{rep1}, we have
\begin{equation}\label{pr1+v}
  \Vert f^*-\V_{m,n}f^*\Vert_{L_p(\T^2)}\le \Vert f^*-\V_{m,\infty}f^*\Vert_{L_p(\T^2)}+\Vert \V_{m,\infty}(f^*-\V_{\infty,n}f^*)\Vert_{L_p(\T^2)}.
\end{equation}
Using Lemma~\ref{lem2v}, item 1, we get
\begin{equation}\label{pr2+v}
  \Vert f^*-\V_{m,\infty}f^*\Vert_{L_p(\T^2)}\le Cm^{-1} E((f^*)^{(1,0)},\Tp_{m,\infty}(L_p))_{L_p(\T^2)},
\end{equation}
where $\Tp_{m,\infty}(L_p)$ is the class of functions $g$ such that $g\in L_p(\T^2)$ and $g$ is a trigonometric polynomial of degree at most $m$ in the first variable.

Let $P_{m-1,\infty}\in \Pp_{m-1,\infty}(L_{p,w})$ be such that
$$
\Vert f^{(1,0)}-P_{m-1,\infty}\Vert_{L_{p,w\vp_1}(J^2)}=E_{m-1,\infty}^\Pp(f^{(1,0)})_{L_{p,w\vp_1}(J^2)}.
$$
Then, by~\eqref{Je2}, it is easy to see that
\begin{equation}\label{P1inf}
  \begin{split}
    E((f^*)^{(1,0)},&\Tp_{m,\infty}(L_p))_{L_p(\T^2)}\\
&\le \Vert \sin\phi f^{(1,0)}(\cos\phi,\cos\psi)-\sin\phi P_{m-1,\infty}(\cos\phi,\cos\psi)\Vert_{L_p(\T^2)}\\
&\le CE_{m-1,\infty}^\Pp(f^{(1,0)})_{L_{p,w\vp_1}(J^2)}\le C n^{-r+1}E_{m-r,\infty}^\Pp(f^{(r,0)})_{L_{p,w\vp_1^{r}}(J^2)}.
  \end{split}
\end{equation}
Thus, combining~\eqref{pr2+v} and~\eqref{P1inf}, we get
\begin{equation}\label{P1inf1}
  \begin{split}
 \Vert f^*-\V_{m,\infty}f^*\Vert_{L_p(\T^2)}\le C n^{-r}E_{m-r,\infty}^\Pp(f^{(r,0)})_{L_{p,w\vp_1^{r}}(J^2)}.
  \end{split}
\end{equation}

Next, by~\eqref{newnew}, we derive
\begin{equation}\label{pr3+v}
  \begin{split}
    \Vert \V_{m,\infty}(f^*-\V_{\infty,n}f^*)\Vert_{L_p(\T^2)}^p&=\int_0^{2\pi} \Vert \V_{m,\infty}(f^*(\cdot,\psi)-\V_{\infty,n}f^*(\cdot,\psi))\Vert_{L_p(\T)}^p {\rm d}\psi\\
&\le C\int_0^{2\pi}\frac1m\sum_{k=0}^{6m-1}|f^*(\phi_k,\psi)-\V_{\infty,n}f^*(\phi_k,\psi)|^p {\rm d}\psi\\
&=\frac Cm\sum_{k=0}^{6m-1}\Vert f^*(\phi_k,\cdot)-\V_{\infty,n}f^*(\phi_k,\cdot)\Vert_{L_p(\T)}^p.
  \end{split}
\end{equation}
By Lemma~\ref{lem2v}, item 1, \eqref{Je2}, and~\eqref{vf*2+}, we get
\begin{equation}\label{pr4+v}
  \begin{split}
    \Vert f^*(\phi_k,\cdot)-\V_{\infty,n}f^*(\phi_k,\cdot)\Vert_{L_p(\T)} &\le Cn^{-1}
E_n^\Tp((f^*)'(\phi_k,\cdot))_{L_p(\T)}\\
&\le Cn^{-s} E_{n-s}^\Pp(f^{(0,s)}(x_k,\cdot))_{L_{p,\vp_2^s w}(J)}.
  \end{split}
\end{equation}

Finally, combining~\eqref{pr1+v}--\eqref{pr4+v}, we obtain~\eqref{v+1}.
\end{proof}

By analogy, using Lemma~\ref{lem2v}, item 2, and equalities~\eqref{vf*1} and~\eqref{vf*2}, one can prove the following result in terms of functions of bounded variation.

\begin{lemma}\label{thV++}
Under the conditions of Theorem~\ref{thL++}, we have for any $m,n\in \N$
\begin{equation*}\label{v++1}
\begin{split}
    \Vert f^*-&\V_{m,n}f^*\Vert_{L_p(\T^2)}\\
&\le C\({m^{-r-1/p}} \Vert V_{1,J}(\widetilde{D}^{(r,0)}f)\Vert_{L_{p,w}(\T)}+{n^{-s-1/p}} \Vert \{V_{2,J} (\widetilde{D}^{(0,s)}f(x_k,\cdot))\}\Vert_{\widetilde{\ell}_p^{6m}}\),
\end{split}
\end{equation*}
where the constant $C$ does not depend on $m$, $n$, and $f$.
\end{lemma}

\begin{remark}
The corresponding results symmetric to Lemma~\ref{thV+} and Lemma~\ref{thV++} can be obtained by using the equality
\begin{equation*}\label{rep2}
\begin{split}
  f^*(\vp,\psi)-&\V_{m,n}f^*(\vp,\psi)\\
&=f^*(\vp,\psi)-\V_{\infty,n}f^*(\vp,\psi)+\V_{\infty,n}\(f^*(\vp,\cdot)-\V_{m,\infty}f^*(\vp,\cdot)\)(\psi)
\end{split}
\end{equation*}
and repeating the proofs of these lemmas. See also Remarks~\ref{rem1} and~\ref{rem2}.
\end{remark}

\subsection{Proofs of the main theorems}

The proofs of Theorem~\ref{thL1}, Theorem~\ref{thL2}, Theorem~\ref{thL+}, and Theorem~\ref{thL++}   follow from Lemma~\ref{pass} and Lemmas~\ref{thVtrig1},~\ref{thVtrig2},~\ref{thV+}, and~\ref{thV++},  correspondingly. Let us also mention that the above Lemmas~\ref{thVtrig1}--\ref{thV++} are valid without the assumption $(m,n)\in \Np^2$.

\subsection{Final remarks}
Note that the above results with minor changes are also valid in the case of the Padua points~\cite{CMV}, the Xu points~\cite{Xu1996}, and the Lissajous-Chebyshev node points of the (non-degenerate) Lissajous curve~\cite{ErbKaethnerAhlborgBuzug2015}. To see this one only needs to note that the above sets of points are subsets of the Chebyshev grid $\{(x_k,y_k)\,:\, (k,l)\in \mathcal{I}_{m,n}'\}$. Then one can apply the corresponding auxiliary results from Section~5.

%

\end{document}